\documentclass[12pt,psfig,reqno]{amsart}
\usepackage{amssymb,amsfonts,latexsym}
\usepackage{graphics,verbatim}
\usepackage{graphicx}
\usepackage[usenames]{color} 
\usepackage{soul} 

\hoffset=-1cm \errorcontextlines=0 \numberwithin{equation}{section}

 \theoremstyle{plain}
\newtheorem{theorem}{Theorem}[section]
\newtheorem{lemma}[theorem]{\bf{Lemma}}
\newtheorem{proposition}[theorem]{\bf{Proposition}}

\date {}

\begin{document}

\title[heat equation and Poisson equation in matrix geometry]{heat equation and Poisson equation in a model matrix geometry}

\author{Jiaojiao Li}

\address{Jiaojiao Li, Department of mathematics \\
Henan Normal university \\
Xinxiang, 453007 \\
China}

\email{lijiaojiao8219@163.com}

\thanks{ The research is partially supported by the National Natural Science
Foundation of China (No.)}

\begin{abstract}
In this paper, we study the Poisson equation and heat equation in a
model matrix geometry $M_n$. Our main results are about the Poisson
equation and global behavior of the heat equation on $M_n$. We can
show that if $c_0$ is the initial positive definite matrix in $M_n$,
then $c(t)$ exists for all time and is positive definite too. We can
also show the entropy stability of the solutions to the heat
equation.


\end{abstract}

\maketitle

\section{Introduction}

  In the Riemannian geometry, the spectrum of the Laplacian on manifold gives the geometric and topological
information about the manifold. The heat equation proof of
Atiyah-Singer index theorem is one of the most famous example
\cite{S}. In particular, through the use of the Heat equation, one
can define the curvature of the compact n-dimensional Riemannian
manifold (M,g) as below. Let H(x,y,t) be the heat kernel of
Laplacian operator \cite{R}. Let $(\lambda_j)$  be the spectrum and
$\{\phi_j(x)\}$ the corresponding eigenfunctions on $M$. Then $$
H(x,y,t)=\sum_je^{-\lambda_jt}\phi_j(x)\phi_j(y).
$$
We have the expansion
$$
(4\pi t)^{n/2}H(x,x,t)=1+\frac{t}{3}R+0(t^2)
$$
near $t=0$.Here $R$ is the scalar curvature of the metric g.  This implies that we can define the scalar curvature by
$$
R=3\frac{d}{dt}|_{t=0}[(4\pi t)^{n/2}H(x,x,t)].
$$
Hence, it is nature to use the heat equation to define scalar
curvature in the non-commutative geometry \cite{C} \cite{C1}\cite{C2}. The aim of the
this paper is to explore this interesting part in a simple case,
which has been recently studied by R.Duvenhage in \cite{D}. In
\cite{D}, the author introduces the Ricci flow and his main result
can be briefly stated as follows. Let $M_n$ be the $C^*$ algebra
generated by the two matrices \cite{Ma}
$$
a=e^{2\pi ix/n}, \ \ b=e^{2\pi iy/n},
$$
where $x,y$ are two Hermitian matrices. Define the derivations
$$
\delta_1:=[y,\cdot], \ \  \delta_2:=-[x,\cdot]
$$
and the Laplacian operator
$$
\Delta=-(\delta_1^2+\delta_2^2).
$$
Then the Ricci flow can be defined by \cite{D}
$$
\frac{d}{dt} c(t)=-\Delta \log c(t).
$$
For any positive definite matrix $c_0\in M_n$, there is a global solution and it converges to the scalar matrix determined
by $c_0$. Furthermore, along this flow, the von Neumann entropy of $c(t)$ is increasing except $c_0$ is a scalar matrix.

We shall introduce in the same $M_n$ the eigenvalues and
eigenfunctions of the Laplacian operator $\Delta$ and define the
heat kernel and the scalar curvature as above. Then we can introduce
the Ricci flow by the standard way that
$$
\frac{d}{dt} c(t)= R(c(t)) c(t)
$$
for the positive definite matric $c(t)$.  This gives the fourth way
to define the Ricci flow in non-commutative geometry. However, since
there is no explicit relation about the scalar curvature and the
matrix $c(t)$, this approach may be very complicated for us to get a
global Ricci flow \cite{H1}\cite{H2} \cite{Fr}\cite{F1}. The Ricci flow found many interesting
applications in physics. It appears as the renormalization group
equations of 2-dimensional sigma models \cite{A} \cite{A1,A2,A3,A4}. It also be used to
study the evolution of the ADM mass in asymptotically flat spaces
\cite{DM}. More recently, it appears in studying the contribution of
black holes in Euclidean quantum gravity \cite{HT}\cite{Ho}. In \cite{B}, the
paper describes an appropriate analog of Hamilton＊s Ricci flow for
the noncommutative two tori, which are the prototype example of
noncommutative manifolds. It is still of interest to find more way
to define the Ricci flow in noncommutative geometry.

Our main results are about the Poisson equation and global behavior
of the heat equation on $M_n$. We can show that if $c_0$ is the
initial positive definite matrix in $M_n$, then $c(t)$ exists for
all time and is positive definite too. We can also show the entropy
stability of the solutions to the heat equation.

\section{elementary noncommutative differential geometry}\label{sect2}

Let $X, Y$ be two Hermitian matrices on $C^{n}$ .
Define $U=e^{\frac{2\pi i}{n}X}$, $V=e^{\frac{2\pi i}{n}Y}$. We use
$M_{n}$ to denote the algebra of all $n\times n$ complex matrices
generated by $U$ and $V$ with the bracket $\{u,v\}=uv-vu$. Then
$CI$, which is the scalar multiples of the identity matrices $I$, is
the commutant of the operation $\{u,v\}$. Sometimes we simply use
$1$ to denote the $n\times n$ identity matrix.

We define two derivations $\delta_{1}$ and $\delta_{2}$ on the
algebra $M_{n}$ by the commutators
$$\delta_{1}:=[y,\cdot]\,\,\, \,\,\,  \delta_{2}:=-[x,\cdot]$$
Define the Laplacian operator on $M_{n}$ by
$$\Delta =\delta_1^*\delta_1+\delta_2^*\delta_2=-\delta_{1}^{2}-\delta_{2}^{2}=-\delta_{\mu}\delta_{\mu},$$
where we have used the Einstein sum convention. We use the
Hilbert-Schmidt norm defined by the inner product
$$<a,b>=\tau(a^{*}b)$$
on the algebra $M_{n}$. Here $a^{*}$ is the complex conjugate of the
matrix $a$, $\tau$ denotes the usual trace function on $M_{n}$. We
now state basic properties of $\delta_{1}$, $\delta_{2}$ and
$\Delta$ (see also \cite{D}).

\begin{proposition}\label{basic pro}
For $a\in M_n$, We have the following properties\\
(a) If $\delta_{1} a=\delta_{2} a=0$, then $a\in CI$. Conversely, $\delta_{1} a=\delta_{2} a=0$.\\
(b) $\tau (a\delta_{\mu}b)=-\tau(b\delta_{\mu }a)$, that is $<a^*, \delta_\mu b>=-<b^{*}, \delta_\mu a>$.\\
Furthermore, if $a,b$ are Hermitian matrices, then
$$
<\delta_\mu^* a,b>=-<\delta_\mu a,b>, \ \ (\delta_\mu a)^*=-\delta_\mu a.
$$
(c) $\exists c>0$, such that
$$c|a-\bar{a}|^{2}\leq <\delta_{\mu}(a-\bar{a},\delta_{\mu}(a-\bar{a} )> \leq c^{-1}|a-\bar{a}|^{2}, $$
where $\bar{a}=\frac{\tau(a)}{n}I$. \\
(d) The operators $-\delta_{1}^{2}$, $-\delta_{2}^{2}$ and $\Delta$
on the Hilbert space $M_{n}$ are positive, that is
$<a,\delta_{\mu}^{2}a>\leq 0$ and $<a,\Delta a>\geq 0$.\\
(e) If $<a, \delta_{\mu}^{2}a>=0$, then $\delta_{\mu}a=0$.\\
(f) $\ker \Delta=CI$. \\
(g) $\tau (\Delta a)=0$.
\end{proposition}

For completeness, we give the detail proof.
\begin{proof}
(a) If $\delta_{1} a=\delta_{2} a=0$, then $[y,a]=-[a,x]=0$, that is $a$  commutes with $x,y$.
so $a$ can commute with the algebra generators $u,v$. Hence $a\in {u,v}=CI$. 

The converse $\delta_{1} a=\delta_{2} a=0$ is trivial.

(b) We only prove the conclusion for $\mu=1$.\\
Compute,
$$\tau(a\delta_1 b)=\tau(a[y,b])=\tau(ayb-aby)
=\tau(bay-bya)$$
$$=-\tau(bya-bay)=-\tau(b[y,a])=-\tau(b\delta_1 a).$$
The similar computation gives result for $\mu=2$.

(c) Define
$$|a-\bar{a}|_{1}=<\delta_{\mu}(a-\bar{a}), \delta_{\mu}(a-\bar{a})>^{\frac{1}{2}}, $$
we verify that $|\cdot|_{1} $ is a norm on $M_{n}/CI$.\\
We need to verify the following three. 

(1) $\bar{a}=0$, if $|a|_{1}=0 \Leftrightarrow\delta_{\mu}(a)=0$, $\mu=1,2.$\\

(2) $\forall\lambda>0$, $|\lambda a|_{1}=|\lambda||a|_{1}$ is clear true.\\

(3) $\forall a, b\in M_{n}/CI$, it is also true that
$$|a+b|_{1}\leq |a|_{1}+|b|_{1}.$$

Since $M_{n}/CI$ is a finite dimension vector space, the Hilbert
Schmidt norm $|\cdot|$ is equivalent to $|\cdot|_{1}$ on $M_{n}/CI$.
 
(d) Compute directly that $$<a,\delta^2_\mu a>=<a,-\delta^*_\mu\delta_\mu a>=-<\delta_\mu a,\delta_\mu a>\leq 0$$ for $\mu=1,2$. 
 The result implies that the positivity of the Laplacian operator.

(e) By the definition of $\delta^2_\mu$, we have
 $$<a, \delta^2_\mu a>=<a,-\delta^*_\mu \delta_\mu a>=-<\delta_\mu a,\delta_\mu a>=0.$$
 So we obtain $\delta_\mu a=0$.

(f) On one hand, for $a\in CI$, by the fact $\Delta a=-\delta^2_1 a-\delta^2_2 a=0$, it is easy to know that $CI\subset \ker\Delta$.
On the other hand, if $a\in \ker \Delta$, that is, $\Delta a=0$, we derive $0=<a, \Delta a>=-<a, \delta^2_\mu>$ by $(d)$,
so $\delta_\mu a=0$ by $(e)$. It follows that $a\in CI$ by $(a)$, so $\ker\Delta \subset CI$.

Therefore $\ker\Delta=CI$.

(g)$\tau(\Delta a)=<1,\Delta a>=<1,-\delta^2_\mu a>=-<\delta_\mu 1,\delta_\mu a>=0$ for $\delta_\mu 1=0.$
\end{proof}

\begin{proposition}
For any positive definite matrix $a\in M_{n}$, $\forall m\in Z$
we have
$$ \tau (a^{m}\Delta a)\geq 0 $$
with equality if and only if $a\in CI$, i.e. if and only if $a$ is a scalar multiple of the
identity matrix $I$.
\end{proposition}

\begin{proof}
By the definitions of $\Delta,\,\delta_1,\,\delta_2 $ and $(b)$ of the above proposition (\ref{basic pro}), we obtain 
$$\tau(a^m \Delta a)=-\tau(a^m \delta^2_\mu a)=-\tau(\delta_\mu a^m \delta_\mu a).$$
Consider that $\tau(\delta_\mu a^m \delta_\mu a)$ is a sum of terms of the form 
\[
\begin{aligned}
\tau((\delta_\mu a) a^p (\delta_\mu a)a^q)
&=\tau(a^{\frac{q}{2}}(\delta_\mu a) a^\frac{p}{2} a^\frac{p}{2}(\delta_\mu a)a^\frac{q}{2})\\
&=-\tau(a^{\frac{q}{2}}(\delta_\mu a)^* a^\frac{p}{2} a^\frac{p}{2}(\delta_\mu a)a^\frac{q}{2})\\
&=-\tau((a^{\frac{q}{2}}(\delta_\mu a) a^\frac{p}{2})^* a^\frac{p}{2}(\delta_\mu a)a^\frac{q}{2})\leq 0.
\end{aligned}
\]
by $\tau(ab)=\tau(ba)$ and the proposition (\ref{basic pro})(b) for $a$ is positive definite, where 
$p,q\in{0,1,2,\cdots}$ with $p+q+1=m$. Therefore, $\tau (a^m \Delta a)\geq 0$.

Now suppose $\tau (a^m \Delta a)=0$. 

Since $\tau(a^m \delta^2_\mu a)\leq 0$, 
it follows that $\tau(a^m\delta^2_\mu a)=0$ for $\mu=1,2$.
In particular, $m=1$ , we derive
$$0=\tau(\delta_\mu a\delta_\mu a)=<(\delta_\mu a)*, \delta_\mu a>=<-\delta_\mu a,\delta_\mu a>=-<\delta_\mu a,\delta_\mu a>.$$
So $\delta_\mu a=0$, hence $a\in CI$ by proposition (\ref{basic pro})(a). The converse is trivial.

The proof is complete.
\end{proof}

\section{Poisson equation }\label{sect3}

We study for a given $b\in M_{n}$, the solvability of the Poisson
equation \begin{equation} \label{P}
\Delta a=b.
\end{equation}
Since
$$<1, \Delta a>=\tau(\Delta a)=\tau(b)=0,$$
we know that the necessary condition to solve the Poisson equation
is $\bar{b}=0$.

We can show that it is also the sufficient condition in the class
$M_n/CI$.
\begin{theorem}\label{lem2.4}
 The Poisson equation (\ref{P}) is solvable in $M_{n}/CI$ if and only if $\bar{b}=0$.
\end{theorem}

\begin{proof}
Assume $\bar{b}=0$, by the result from linear algebra, we only need
to show the homogeneous equation
$$\Delta a=0$$
has only zero solution in  $M_{n}/CI$. In fact, if $\Delta a=0$, then
$$<a,\Delta a>=0.$$
Note $<a,\Delta a>=-<a, \delta^{2}_{1} a>-<a, \delta^{2}_{2} a>=0$, then $<a, \delta^{2}_{\mu}a>=0$. \\
Hence, 
$\delta_{\mu}a=0$. Then $a\in CI$.
\end{proof}
We also have the following result for the eigenvalue of the
Laplacian operator.

\begin{lemma}
\begin{eqnarray}
\lambda_{1}=\inf \{\frac{<\Delta a,a>}{<a,a>},a\not=0,\}
\end{eqnarray}
\end{lemma}
is the least eigenvalue of $\Delta$ on $M_{n}/CI$.

\begin{proof}
Let $a_{n}\in M_{n}/CI$, $|a_{n}|=1$, such that
$$<\Delta a_{n},a_{n}> \rightarrow \lambda_{1}.$$

By Weierstrass compactness theorem for bound sequences in finite
vector space, we may assume
$$|a_{n}-a_{\infty}|\rightarrow 0.$$
Hence, $|a_{\infty}|=1$ and $<\Delta a_{n},a_{n}>\rightarrow <\Delta a_{\infty},a_{\infty}>=\lambda_{1}$.

By variational principal,
$$\Delta a_{\infty}=\lambda_{1}a_{\infty}$$
That is to say $a_{\infty}$ is the eigenvalue corresponding to the least eigenvalue $\lambda_{1}$.
\end{proof}

\section{Heat equation}\label{sect4}

 In this section we study the heat equation
 \begin{equation} \label{H}
 u_{t}=-\Delta u,\,\,\,\,u\in{M_{n}},
\end{equation}
$$u|_{t=0}=u_{0}.$$

Since (\ref{H}) is an ODE, it has a local solution $u=u(t)$.\\
Note $\bar{u}_{t}=\frac{d}{d_{t}}\tau(u)=\tau(u_{t})=-\tau(\Delta u)=0$.
\[
\begin{aligned}
\frac{d}{d_{t}}|u-\bar{u}|^{2}
& =2<u-\bar{u},u_{t}>\\
& =-2<u-\bar{u},\Delta u>\\
& =-2<\delta_{\mu} u, \delta_{\mu} u>\\
& \leq-2C|u-\bar{u}|^{2}
\end{aligned}
\]
Since $|u-\bar{u}|^{2}\leq Ae^{-2Ct}\rightarrow 0$ as $t\rightarrow \infty$,
then $(\ref{H})$ has a global solution and $\bar{u}=\lim_{t\rightarrow \infty} u(t)=\bar{u}_{0}$.\\

Assume $u_{0}>0$, we claim $u(t)>0$. 

Since
\[
\begin{aligned}
\frac{d}{d_{t}}\det u &=\tau (u^{-1}u_{t})\\
&=-\tau (u^{-1}\Delta u)\\
&=-<u^{-1}, \Delta u>\\ 
&=-<\delta_{\mu} u^{-1}, \delta_{\mu} u>.
\end{aligned}
\]
Note
$$\delta_{1}u^{-1}=[y, u^{-1}]=yu^{-1}-u^{-1}y=u^{-1}(uy-yu)u^{-1}=-u^{-1}\delta_{1}u u^{-1}.$$
Then
\[
\begin{aligned}
\frac{d}{d_{t}}\det u &=<u^{-1}\delta_{1}u u^{-1},\delta_{\mu}u>\\
&=\tau (u^{-1}\delta_{1}u u^{-1}\delta_{\mu}u)\\
&=<u^{-1}\delta_{\mu}u, u^{-1}\delta_{\mu}u>\\
&>0. 
\end{aligned}
\]
That is to say, $\det u $ is increasing function, hence $\det u>0$. So $u(t)>0, $ for $ t>0$.

In conclusion, we have proven
\begin{theorem}
For any $u_{0}\in M_{n}$, $(\ref{H})$ has a global solution $u(t)$ with its limit $\bar u_{0}$.
Furthermore, if $u_{0}>0$, then $u(t)>0, \forall t>0$.
\end{theorem}
In below, we assume $u_{0}>0$ and define the von Neumann entropy by
$$S(u)=-\tau (u\log u)$$
for the positive solution $u=u(t)$ with $u(0)=u_{0}$.\\
We have the following result

\begin{proposition}
The entropy $S(u)$ is increasing along the heat equation $(\ref{H})$.
\end{proposition}

\begin{proof}
\[
\begin{aligned}
\frac{d}{d_{t}}S(u)=-\tau (u_{t}\log u)-\tau (uu^{-1}u_{t})= \tau (\Delta u \log u) = \tau ( u \Delta \log u).
\end{aligned}
\]
Set $l=\log u$, then $u=e^{l}$.

So
$$\frac{d}{d_{t}} S(u)=\tau (e^{l}\Delta l)\geq 0 .$$
\end{proof}

\section{Entropy stability of the heat equation }\label{sect5}

Given two initial matric $u_{0}$, $v_{0}$. Let $u, v$ be the corresponding solutions.

\begin{proof}
\[
\begin{aligned}
\frac{d}{d_{t}}|u-v)|^{2} & = 2<u-v,\Delta u-\Delta v>\\
&= 2<u-v, \Delta (u-v)>\\
&\leq -2c|u-v|^{2}.
\end{aligned}
\]
\end{proof}
$$|u-v|^{2}\leq Ae^{-2Ct}\rightarrow 0,$$
where $A=|u-v|^{2}(0)$.
Remark: Similarly, we have the Trace norm stability of solutions, where the Trace norm is denoted by $T(u,v)$.
This implies the Hilbert Schmidt norm stability of equation $(\ref{H})$.\\

Recall the Fannes inequality for $\forall a, b \in M_{n}$ and $a>0\,\, b>0$, we have
$$|S(a)-S(b)|\leq \hat{\Delta}\log d +\eta (\hat{\Delta}),$$
where $\eta (s)=-s\log s$, $d=dim M_{n}$ and
$\hat{\Delta}=\sum |r_{i}-s_{i}|\leq T(a,b)\leq \frac{1}{e}.$

Then we can use the Fannes inequality to get the entropy stability
of the solution of $(\ref{H})$.
\begin{theorem}
If $T(u_{0},v_{0})\leq\frac{1}{e}$, $u_{0}>0\,\, v_{0}>0$ in $M_{n}$, then
the solution $u(t), v(t)$ satisfies
$$|S(u_{t})-S(v_{t})|\leq T(u, v)(0)\log d+\eta (T(u,v))(0).$$
\end{theorem}

\begin{proof}
By the result above we have
$$T(u, v)(t)\leq T(u, v)(0),$$
by the Fannes inequality, we have
\[
\begin{aligned}
|S(u_{t})-S(v_{t})| & \leq T(u(t), v(t))\log d +\eta (T(u(t),v(t))\\
& \leq T(u(0),v(0))\log d+\eta(T(u(0), v(0))),
\end{aligned}
\]
where we have used the monotonicity of the function $\eta$ in $[0, \frac{1}{e}]$.
\end{proof}

\end{document}